\documentclass[12pt]{amsart}
\usepackage{dsfont}
\usepackage{amssymb, amsmath, amsthm}
\usepackage{mathrsfs}
\newcommand{\no}[1]{#1}
\renewcommand{\no}[1]{}
\no{\usepackage{times}\usepackage[subscriptcorrection, slantedGreek, nofontinfo]{mtpro}
\renewcommand{\Delta}{\upDelta}}
\usepackage{color}
\usepackage{graphicx}
\usepackage{enumerate,comment}

\newtheorem{theorem}{Theorem}[section]
\newtheorem{prop}{Proposition}[section]
\newtheorem{lem}{Lemma}[section]

\theoremstyle{remark}

\newcommand{\bel}{\begin{equation} \label}
\newcommand{\ee}{\end{equation}}

\newcommand{\R}{{\mathbb R}}

\def\phi {\varphi}

\renewcommand{\leq}{\leqslant}
\renewcommand{\geq}{\geqslant}

\def\beq{\begin{equation}}
\def\eeq{\end{equation}}
\newcommand{\bea}{\begin{eqnarray}}
\newcommand{\eea}{\end{eqnarray}}
\newcommand{\beas}{\begin{eqnarray*}}
\newcommand{\eeas}{\end{eqnarray*}}

\providecommand{\abs}[1]{\left\lvert#1\right\rvert}
\providecommand{\norm}[1]{\left\lVert#1\right\rVert}

\numberwithin{equation}{section}

\title[Stable  determination of nonlinear terms]{Lipschitz and H\"older stable  determination of nonlinear terms for  elliptic equations}
\begin{document}
\begin{abstract} We consider the inverse problem of determining some class of nonlinear terms appearing in an elliptic equation from boundary measurements. More precisely, we study the stability issue for this class of inverse problems. Under suitable assumptions, we prove a Lipschitz and a H\"older stability estimate associated with the determination of quasilinear and semilinear terms appearing in this class of elliptic equations from measurements restricted to an arbitrary parts of the boundary of the domain. Besides their mathematical interest, our stability estimates can be useful for improving numerical reconstruction of this class of nonlinear terms.  Our approach combines the linearization technique with applications of suitable class of singular solutions.


\end{abstract}


\author[Yavar Kian]{Yavar Kian}
\address{Aix Marseille Univ, Universit\'e de Toulon, CNRS, CPT, Marseille, France.}
\email{yavar.kian@univ-amu.fr}

\maketitle


\section{Introduction}

\subsection{Statement}
Let $\Omega$ be a bounded domain of $\mathbb{R}^n$, $n\geq 2$, with $\mathcal C^{2+\alpha}$, $\alpha\in(0,1)$, boundary.  Let  $\gamma\in \mathcal C^3(\R;(0,+\infty))$, $F\in \mathcal C^2(\overline{\Omega}\times\R\times\R^n;\R)$ and let $a:=(a_{i,j})_{1 \leq i,j \leq n} \in \mathcal C^2(\overline{\Omega};\R^{n^2})$
be symmetric, that is 
$$ 
a_{i,j}(x)=a_{j,i}(x),\ x \in \Omega,\ i,j = 1,\ldots,n, 
$$
and let $a$ fulfill the ellipticity condition: there exists a constant
$c>0$ such that
\bel{ell}
\sum_{i,j=1}^d a_{i,j}(x) \xi_i \xi_j \geq c |\xi|^2, \quad 
\mbox{for each $x \in \overline{\Omega},\ \xi=(\xi_1,\ldots,\xi_n) \in \R^n$}.
\ee
We consider the following boundary value problem
\bel{eq1}
\left\{
\begin{array}{ll}
-\sum_{i,j=1}^d \partial_{x_i} 
\left( a_{i,j}(x)\gamma(u(x)) \partial_{x_j} u(x) \right)+F(x,u(x),\nabla u(x))=0  & \mbox{in}\ \Omega ,
\\
u=g &\mbox{on}\ \partial\Omega,
\end{array}
\right.
\ee
with   $g\in \mathcal C^{2+\alpha}(\partial\Omega)$. We assume here  that the nonlinear terms $\gamma$, $F$  satisfy one of the following conditions:\\
 (i) $F$ takes the form
\bel{non1} F(x,u(x),\nabla u(x))=D(x,u(x))\cdot\nabla u(x),\quad x\in\Omega,\ee
with  $D\in \mathcal C^2(\overline{\Omega}\times\R;\R^n)$.\\
 (ii) $\gamma=1$ and
\bel{non2} \quad F(x,u(x),\nabla u(x))=G(u(x)),\quad x\in\Omega,\ee
with $G\in \mathcal C^2(\R)$ a non-decreasing function.

Under some suitable assumptions, that will be specified later, we prove that the problem \eqref{eq1}  admits a unique and sufficiently smooth solution $u_g$. Then, fixing $S$ an arbitrary open subset of $\partial\Omega$ we associate with problem \eqref{eq1}  the partial Dirichlet-to-Neumann (DN in short) map
$$\mathcal N_{\gamma,F}:g\mapsto \gamma(u_g)\partial_{\nu_a} u_g|_{S}$$
with  $\nu(x)=(\nu_1(x),\ldots,\nu_n(x))$ the outward unit normal to $\partial\Omega$ computed at $x \in \partial\Omega$ and
$$\partial_{\nu_a} u(x)=\sum_{i,j=1}^na_{i,j}(x)\partial_{x_j}u(x)\nu_i(x),\quad x\in\partial\Omega.$$
We study the inverse problem of determining the nonlinear term $\gamma$  when condition (i) is fulfilled or $G$  when condition (ii) is fulfilled from some knowledge of $\mathcal N_{\gamma,F}$. More precisely, we are looking for stability result for this inverse problem.
\subsection{Motivations}

Let us observe that the nonlinear equations under consideration in \eqref{eq1} can be associated with different physical phenomenon which can not be described by classical linear elliptic equations. For instance, we can mention  physical models  of viscous flows \cite{GR} or plasticity phenomena \cite{CS} where the transfer from voltage to current density can not be described by the classical Ohm's law but some more general nonlinear expression with a conductivity $\gamma$ depending on the voltage. The solutions of problem \eqref{eq1} can also be seen as  stationary solutions of nonlinear Fokker-Planck or reaction-diffusion equations where the nonlinearity can arise from cooperative interactions between the subsystems of many-body systems (see e.g. \cite{F}), complex mixing phenomena (e.g. \cite{NW}) or models appearing in  combustion theory (e.g. \cite{ZF}). In this context, the goal of our inverse problem is to determine the nonlinear law associated with the quasilinear term $\gamma$ or the semilinear term $F$ from measurements of the flux, localized at some arbitrary parts of the boundary $\partial\Omega$, associated with different excitation of the system (voltage, heat source...) applied at the boundary $\partial\Omega$. More precisely, we are looking for Lipschitz or H\"older stability estimates for this class of inverse problems which can be useful tools for improving the precision of   numerical reconstruction (see e.g. \cite{CY}).
\subsection{Known results}

The determination of nonlinear terms appearing in elliptic equations is an important class of inverse problems which has been intensively studied these last decades. One of the first important results for these problems can be found in \cite{Is5,Is4} (see also \cite{ImYa}) where the authors considered the unique determination of a general semilinear term of the form $F(x,u)$, $x\in\Omega$, $u\in\R$, by applying the method of linearization initiated by \cite{Is1}. This approach has been extended to the unique determination of quasilinear terms  by \cite{EPS1,EPS2,MU,Sh,Sun1,SuUh} and to the determination of semilinear terms of the form $F(u,\nabla u)$ by \cite{Is2,Is3}. More recently, this class of inverse problems received an important attention with different contributions obtained by mean of the multiple order linearization approach introduced by \cite{KLU}. In that category, without being exhaustive, we can mention the works of \cite{FO20,KKU,KU0,LLLS} who have considered the unique determination of semilinear terms of the form $F(x,u)$, $x\in\Omega$, $u\in\R$ and the works of \cite{CFKKU,KKU} dealing with the unique determination of general quasilinear terms of the form $\gamma(x,u,\nabla u)$, $x\in\Omega$. 

All the above mentioned results have been stated in terms of uniqueness results. In contrast to the important development of this topic in terms of uniqueness results, only few authors considered the stability issue for these problems. For elliptic equations, we are only aware of the works  \cite{CHY,LLLS} where the authors proved a logarithmic stability estimate for the determination of semilinear terms of the form  $F(u)$, $u\in\R$, with $F$ a sufficiently smooth function or  semilinear terms of the form $F(x,u)=q(x)u^m$,  $x\in\Omega$, $u\in\R$, where $m$ is a known integer and the parameter $q$ is unknown. We mention also the recent work of \cite{LLPT} where the authors studied the stable determination of semilinear terms appearing in nonlinear hyperbolic equations. To the best of our knowledge there is no result available in the mathematical literature showing Lipschitz or H\"older stability estimate for the determination of nonlinear terms appearing in nonlinear elliptic equations.

\subsection{Main results}
In order to state our main results we will start by recalling some properties of the solution $u$ of \eqref{eq1} and the associated partial DN map $\mathcal N_{\gamma,F}$. Assuming that condition (i) is fulfilled, we prove in  Proposition \ref{p1},  that for all constant $\lambda\in\R$ there exists $\epsilon_\lambda>0$, depending on $a$, $\gamma$, $F$, $\lambda$ and $\Omega$, such that for $g=\lambda+h$, with $h\in \{f\in\mathcal C^{2+\alpha}(\partial\Omega):\ \norm{h}_{\mathcal C^{2+\alpha}(\partial\Omega)}<\epsilon_\lambda\}$, problem \eqref{eq1} admits a unique solution $u_g\in \mathcal C^{2+\alpha}(\overline{\Omega})$. 
In addition, we show in Lemma \ref{l1} that the partial DN map $\mathcal N_{\gamma,F}$ admits a Fr\'echet derivative  at $\lambda$ which allows us to consider the map
\bel{map}\mathcal C^{2+\alpha}(\partial\Omega)\ni h\mapsto\lim_{\epsilon\to0}\frac{\mathcal N_{\gamma,F}(\lambda+\epsilon h)-\mathcal N_{\gamma,F}(\lambda)}{\epsilon}\ee
as a bounded linear map from $\mathcal C^{2+\alpha}(\partial\Omega)$ to $\{f|_S:\ f\in\mathcal C^{1+\alpha}(\partial\Omega)\}$. We prove also that this map can be extended uniquely by density to a continuous linear map from $ H^{\frac{1}{2}}(\partial\Omega)$ to $H^{-\frac{1}{2}}(S)$. Using this property and fixing $\mathcal H_S$ the subset of $\mathcal C^{2+\alpha}(\partial\Omega)$ defined by
$$\mathcal H_S:=\{f\in \mathcal C^{2+\alpha}(\partial\Omega):\ \norm{f}_{H^{\frac{1}{2}}(\partial\Omega)}\leq 1,\ \textrm{supp}(f)\subset S\},$$
 we show in Section 2 that the map
\bel{mesure}\R\ni\lambda\mapsto\sup_{h\in\mathcal H_S}\limsup_{\epsilon\to0}\frac{\norm{\mathcal N_{\gamma,F}(\lambda+\epsilon h)-\mathcal N_{\gamma,F}(\lambda)}_{H^{-\frac{1}{2}}(S)}}{\epsilon}\ee
is lying in $\mathcal C(\R)$. We will use this last map in terms of measurements  for our first stability result, where we treat the determination of the quasilinear term $\gamma$ from the above knowledge of $\mathcal N_{\gamma,F}$.

\begin{theorem}\label{t1}   For $j=1,2$, let $\gamma_j\in \mathcal C^3(\R;(0,+\infty))$ and $D_j\in C^2(\overline{\Omega}\times\R)$ and consider $F_j$ defined by \eqref{non1} with $D=D_j$. We assume that
\bel{t11a}\nabla_x\cdot D_j(x,t)\leq 0,\quad x\in\Omega,\ t\in\R,\ j=1,2.\ee
Then, fixing $\mathcal N=\mathcal N_{\gamma_1,F_1}-\mathcal N_{\gamma_2,F_2}$,  there exists a constant $C>0$, depending only on $a$, $\Omega$, $S$  such that the following estimate 
\bel{t1c}\norm{\gamma_1-\gamma_2}_{L^\infty(-R,R)}\leq C\sup_{\lambda\in[-R,R]}\sup_{h\in\mathcal H_S}\limsup_{\epsilon\to0}\frac{\norm{\mathcal N(\lambda+\epsilon h)-\mathcal N(\lambda)}_{H^{-\frac{1}{2}}(S)}}{\epsilon},\quad R>0,\ee
holds true.\end{theorem} 

For our second result we assume that condition (ii) is fulfilled. We consider also $S'$ an open subset of $\partial\Omega$, to be defined later (see Section 3.1), whose closure is contained into $S$ and we fix $\chi\in\mathcal C^{2+\alpha}(\partial\Omega)$ such that supp$(\chi)\subset S$ and $\chi=1$ on $S'$. Following \cite[Theorem 8.3, pp. 301]{LU} and \cite[Theorem 9.3, 9.4]{GT}, we can show that for $g=\lambda\chi+h$, with $h\in \mathcal C^{2+\alpha}(\partial\Omega)$ and $\lambda\in\R$ a constant, problem \eqref{eq1} admits a unique solution $u_g\in \mathcal C^{2+\alpha}(\overline{\Omega})$.  In addition, in a similar way as above, applying Lemma \ref{l2}, we can prove that the map
\bel{mesure1}\R\ni\lambda\mapsto\sup_{h\in\mathcal H_S}\limsup_{\epsilon\to0}\frac{\norm{\mathcal N_{\gamma,F}(\lambda\chi+\epsilon h)-\mathcal N_{\gamma,F}(\lambda\chi)}_{H^{-\frac{1}{2}}(S)}}{\epsilon}\ee 
is lying in $\mathcal C(\R)$. We will use this last map for our second result where we prove the stable determination of the semilinear term $G$ from the knowledge of $\mathcal N_{\gamma,F}$ given by \eqref{mesure1}.

\begin{theorem}\label{t2}  We assume that $n\geq3$ and that
\bel{t2b}a_{i,j}(x)=\delta_{ij},\quad i,j=1,\ldots,n,\ x\in \Omega,\ee
with $\delta$ the Kronecker delta symbol. For $j=1,2$, we fix $G_j\in \mathcal C^2(\R;\R)$ a non-decreasing function and we consider $F_j$ defined by \eqref{non2} with $G=G_j$.  We assume that  there exists a non-decreasing function $\kappa\in \mathcal C(\R_+)$ such that
\bel{t2a}G_1(0)=G_2(0),\quad \sum_{i,j=1}^2|G_j^{(i)}(\lambda)|\leq \kappa(|\lambda|),\quad \lambda\in\R.\ee
Then, fixing $\mathcal N=\mathcal N_{1,F_1}-\mathcal N_{1,F_2}$ and $R>0$,  we can find a constant $C_R>0$, depending on  $\Omega$, $\chi$, $\kappa$, $R$,  such that the following estimate 
\bel{t2c}\norm{G_1-G_2}_{L^\infty(-R,R)}\leq C_R\left(\sup_{\lambda\in[-R,R]}\sup_{h\in\mathcal H_S}\limsup_{\epsilon\to0}\frac{\norm{\mathcal N(\lambda\chi+\epsilon h)-\mathcal N(\lambda \chi)}_{H^{-\frac{1}{2}}(S)}}{\epsilon}\right)^{\frac{1}{3}}\ee
holds true.\end{theorem} 

Let us observe that the estimate \eqref{t1c} in  Theorem \ref{t1} is a Lipschitz stability estimate in the determination of the quasilinear term $\gamma(u)$ appearing in \eqref{eq1} from some knowledge of the associated partial DN map $\mathcal N_{\gamma,F}$. The result of Theorem \ref{t1} can be seen as the determination of the nonlinear diffusion term $\gamma(u)$ when the nonlinear drift vector $D(x,u)$ is unknown for nonlinear stationary Fokker–Planck equations. Not only the stability estimate \eqref{t1c} is established independently of the choice of the convection term $D_j(x,u)$, $j=1,2$, but the constant of this stability estimate is completely independent of the nonlinear terms $\gamma_j$. This means that the result of Theorem \ref{t1} is not a conditional stability estimate requiring any \textit{a priori} estimate of the unknown parameter. In addition, the result of Theorem \ref{t1} is stated with variable second order coefficients and the measurements are restricted to any arbitrary open subset $S$ of the boundary $\partial\Omega$. 
To the best of our knowledge, we obtain in  Theorem \ref{t1} the first result of Lipschitz stability in the determination of a nonlinear term stated in such a general context and with measurements restricted to an arbitrary open subset of the boundary.

In contrast to Theorem \ref{t1}, the result of Theorem \ref{t2} is a H\"older stability estimate for the determination of a semilinear term $G$ when condition (ii) is fulfilled. This results is stated in a more restricted context than Theorem \ref{t1} and in the final stability estimate \eqref{t2c} the constant depend  on an \textit{a priori} estimate of the nonlinear term $G_j$, $j=1,2$, given by condition \eqref{t2a}. However, in contrast to Theorem \ref{t1}, in Theorem \ref{t2} we restrict both the support of the Dirichlet data $g$ appearing in \eqref{eq1} and the measurements to an arbitrary subset $S$. 

Let us mention that the proof Theorem \ref{t1} and \ref{t2} are based on a suitable application of linearization technique combined with applications of singular solutions suitably designed for our problem. This approach, allows us to obtain Lipschitz and H\"older stability estimate for this inverse problem while, as far as we know, all other results in that category were restricted to logarithmic stability estimate. We recall that this improvement of the stability estimate and the fact that in \eqref{t1c} the constant is independent of the size of the unknown parameter can be exploited for numerical reconstruction of these parameters by mean of Tikhonov's regularization approach (see e.g. \cite{CY}). In some near future, we plan to exploit the material  of this article and some further results for proving the numerical reconstruction of nonlinear terms.

Let us remark that, since the goal of our inverse problems is to determine unbounded functions defined on $\R$, our stability estimates need to be localized. This is why the stability estimates \eqref{t1c} and  \eqref{t2c} are stated in terms of estimates of the nonlinear terms on an interval of the form $(-R,R)$ for $R>0$ arbitrary chosen. While in \eqref{t1c} the constant is independent of $R$, the constant of the stability estimate \eqref{t2c} depends on $R$. We mention also that our stability estimates are stated with some knowledge of the DN map that can be seen as the limit value of the ratio of the differences of the map $\mathcal N_{\gamma,F}$. This result can be formulated in a similar way from the first order Fr\'echet derivative of $\mathcal N_{\gamma,F}$ considered in several articles treating the same issue (see e.g. \cite{CHY,CK}). In contrast to this last formulation, the stability estimates \eqref{t1c} and  \eqref{t2c} are stated with measurements   given by some explicit knowledge of  $\mathcal N_{\gamma,F}$  applied to some class of Dirichlet data.

This article is organized as follows. In Section 2, we recall some properties of \eqref{eq1} including the well-posdness of this boundary value problem and the linearization of our inverse problems. In Section 3, we introduce some class of singular solutions associated to our linearized problem and some of their properties. Finally,  Section 4 will be devoted to the completion of the proof of Theorem \ref{t1} and \ref{t2}.
\section{Preliminary properties}
In this section we consider several properties including the proof of the well-posedness of \eqref{eq1} in the context of Theorem \ref{t1} and \ref{t2}. We will also introduce the linearization of the problem \eqref{eq1} when condition (i) or condition (ii) is fulfilled.
We start by proving the well-posedness of \eqref{eq1} when the Dirichlet data $g$ takes the form $\lambda+h$ with $\lambda\in\R$ a constant and $h\in\mathcal C^{2+\alpha}(\partial\Omega)$  a sufficiently small Dirichlet data. This result can be stated as follows.

\begin{prop}\label{p1} Let condition (i) be fulfilled and let $g=\lambda+h$ with $\lambda\in\R$ a constant and $h\in\mathcal C^{2+\alpha}(\partial\Omega)$. Then, for all $\lambda\in\R$, there exists $\epsilon_\lambda>0$ depending on  $a$, $\gamma$, $D$, $\lambda$, $\Omega$,  such that, for $h\in  B_{\epsilon_\lambda}:=\{f\in\mathcal C^{2+\alpha}(\partial\Omega):\ \norm{f}_{\mathcal C^{2+\alpha}(\partial\Omega)}<\epsilon_\lambda\}$, problem \eqref{eq1} admits a unique solution $u_\lambda\in\mathcal C^{2+\alpha}(\overline{\Omega})$ satisfying
\bel{TA1a}\norm{u_\lambda-\lambda}_{\mathcal C^{2+\alpha}(\overline{\Omega})}\leq C\norm{h}_{\mathcal C^{2+\alpha}(\partial\Omega)}.\ee
\end{prop}

\begin{proof} We prove this result by extending the approach of \cite[Theorem B.1.]{CFKKU}, where a similar boundary value problem has been studied with infinity smooth parameters and  without the convection term $D$, to problem \eqref{eq1}. Let us first observe that constant functions are solutions of \eqref{eq1} with $g=\lambda$. Thus,  by splitting $u_\lambda$ into two terms $u_\lambda=\lambda+u_\lambda^0$, we deduce that $u_\lambda^0$ solves
\bel{eq4}
\left\{
\begin{array}{ll}
-\sum_{i,j=1}^d \partial_{x_i} 
\left( a_{i,j}(x)\gamma(\lambda+u_\lambda^0) \partial_{x_j} u_\lambda^0(x) \right)+D(x,u_\lambda^0(x)+\lambda)\cdot\nabla u_\lambda^0=0  & \mbox{in}\ \Omega ,
\\
u_\lambda^0=h &\mbox{on}\ \partial\Omega,
\end{array}
\right.
\ee
Therefore, we only need to prove that there exists $\epsilon_\lambda>0$ depending on  $a$, $\gamma$, $D$, $\lambda$, $\Omega$, such that, for $h\in  B_{\epsilon_\lambda}$, problem \eqref{eq4} admits a unique solution $u_\lambda^0\in\mathcal C^{2+\alpha}(\overline{\Omega})$ satisfying
\bel{TA1v}\norm{u_\lambda^0}_{\mathcal C^{2+\alpha}(\overline{\Omega})}\leq C\norm{h}_{\mathcal C^{2+\alpha}(\partial\Omega)}.\ee
For this purpose, we introduce the map $\mathcal G$ from $\mathcal C^{2+\alpha}(\partial\Omega)\times\mathcal C^{2+\alpha}(\overline{\Omega})$ to the space $C^{\alpha}(\overline{\Omega})\times\mathcal C^{2+\alpha}(\partial\Omega)$ defined by
$$\mathcal G: (f,v)\mapsto\left(-\sum_{i,j=1}^d \partial_{x_i} 
\left( a_{i,j}(x)\gamma(\lambda+v) \partial_{x_j} v \right)+D(x,v+\lambda)\cdot\nabla v, v_{|\partial\Omega}-f\right).$$
 Using the fact that $\gamma\in \mathcal C^{3}(\R)$ and applying \cite[Theorem  A.8]{H}, we deduce that for any $v\in \mathcal C^{2+\alpha}(\overline{\Omega})$ we have $x\mapsto\gamma'(\lambda+v(x))\in \mathcal C^{1+\alpha}(\overline{\Omega})$. Then, applying \cite[Theorem  A.7]{H}, we deduce that 
$$x\mapsto-\sum_{i,j=1}^d \partial_{x_i} 
\left( a_{i,j}(x)\gamma'(\lambda+v(x)) \partial_{x_j} v(x)\right)\in \mathcal C^{\alpha}(\overline{\Omega})$$
and the map
$$\mathcal C^{2+\alpha}(\overline{\Omega})\ni v\mapsto -\sum_{i,j=1}^d \partial_{x_i} \left( a_{i,j}(x)\gamma(\lambda+v) \partial_{x_j} v\right)\in \mathcal C^{\alpha}(\overline{\Omega})$$
is $\mathcal C^1$.
In the same way, using the fact that $D\in \mathcal C^2(\overline{\Omega}\times\R;\R^n)$, we find
$$x\mapsto[\partial_t D(x,t)|_{t=v(x)+\lambda}]\cdot\nabla v(x)\in \mathcal C^{\alpha}(\overline{\Omega})$$
and we deduce that the map
$$\mathcal C^{2+\alpha}(\overline{\Omega})\ni v\mapsto D(x,v+\lambda)\cdot\nabla v\in \mathcal C^{\alpha}(\overline{\Omega})$$
is $\mathcal C^1$. It follows  that the map $\mathcal G$ is  $\mathcal C^1$ from $\mathcal C^{2+\alpha}(\partial\Omega)\times\mathcal C^{2+\alpha}(\overline{\Omega})$ to the space $C^{\alpha}(\overline{\Omega})\times\mathcal C^{2+\alpha}(\partial\Omega)$. Moreover, we have $\mathcal G(0,0)=(0,0)$ and
$$\partial_v\mathcal G(0,0)w=\left(-\gamma(\lambda)\sum_{i,j=1}^d \partial_{x_i} 
\left( a_{i,j}(x) \partial_{x_j} w\right)+D(x,\lambda)\cdot\nabla w, w_{|\partial\Omega}\right).$$
In view of \cite[Theorem 6.8]{GT} and the fact that $\gamma(\lambda)>0$, for any $(F,f)\in \mathcal C^{\alpha}(\overline{\Omega})\times\mathcal C^{2+\alpha}(\partial\Omega)$, the following linear boundary value problem
$$
\left\{
\begin{array}{ll}
-\gamma(\lambda)\sum_{i,j=1}^d \partial_{x_i} 
\left( a_{i,j}(x) \partial_{x_j} w \right)+D(x,\lambda)\cdot\nabla w=F  & \mbox{in}\ \Omega ,
\\
w=f &\mbox{on}\ \partial\Omega.
\end{array}
\right.$$
admits a unique solution $w\in \mathcal C^{2+\alpha}(\overline{\Omega})$ satisfying
$$\norm{w}_{\mathcal C^{2+\alpha}(\overline{\Omega})}\leq C(\norm{F}_{\mathcal C^{\alpha}(\overline{\Omega})}+\norm{f}_{\mathcal C^{2+\alpha}(\partial\Omega)}),$$
with $C>0$ depending only on $a$, $\lambda$, $\Omega$ and $D$. Thus, $\partial_v\mathcal G(0,0)$ is an isomorphism from $\mathcal C^{2+\alpha}(\overline{\Omega})$ to $\mathcal C^{\alpha}(\overline{\Omega})\times\mathcal C^{2+\alpha}(\partial\Omega)$ and, applying the implicit function theorem, we deduce that there exists $\epsilon_\lambda>0$ depending on  $a$, $\lambda$, $\Omega$ and $D$, and a $\mathcal C^1$ map $\psi$ from $ B_{\epsilon_\lambda}$ to $\mathcal C^{2+\alpha}(\overline{\Omega})$, such that, for all $f\in  B_{\epsilon_\lambda}$, we have 
$\mathcal G(f,\psi(f))=(0,0)$.
This proves that, for all $f\in  B_{\epsilon_\lambda}$, $v=\psi(f)$ is a solution of \eqref{eq4}. Recalling that a solution of the problem   \eqref{eq4} can also be seen as a solution of the linear problem with sufficiently smooth coefficients depending on $u^0_\lambda$, we can apply again \cite[Theorem 6.8]{GT} in order to deduce that $u^0_\lambda=\psi(f)$ is the unique solution of \eqref{eq4}. Combining this with the fact that $\psi$ is $\mathcal C^1$ from $B_{\epsilon_\lambda}$ to $\mathcal C^{2+\alpha}(\overline{\Omega})$ and $\psi(0)=0$, we obtain \eqref{TA1v}. This completes the proof of the proposition.\end{proof}
In view of Proposition \ref{p1}, for all constant $\lambda\in\R$, we can define $\mathcal N_{\gamma,F}$ on  the set $\lambda+B_{\epsilon_\lambda}=\{\lambda+h:\ h\in B_{\epsilon_\lambda}\}$.

Now let us fix $\lambda\in\R$, $h\in\mathcal C^{2+\alpha}(\partial\Omega)$ and  $s_h=\frac{\epsilon_\lambda}{2\norm{h}_{\mathcal C^{2+\alpha}(\partial\Omega)}+1}$. For $s\in[-s_h,s_h]$, we consider the following boundary value problem 
\bel{eq5}
\left\{
\begin{array}{ll}
-\sum_{i,j=1}^d \partial_{x_i} 
\left( a_{i,j}(x)\gamma(u_{\lambda,s}) \partial_{x_j} u_{\lambda,s} \right)+D(x,u_{\lambda,s})\cdot\nabla u_{\lambda,s}=0  & \mbox{in}\ \Omega ,
\\
u_{\lambda,s}=\lambda+sh &\mbox{on}\ \partial\Omega.
\end{array}
\right.
\ee
We consider also the solution of the following linear problem
\bel{eq6}
\left\{
\begin{array}{ll}
-\gamma(\lambda)\sum_{i,j=1}^d \partial_{x_i} 
\left( a_{i,j}(x) \partial_{x_j} w_\lambda \right)+D(x,\lambda)\cdot\nabla w_\lambda=0  & \mbox{in}\ \Omega ,
\\
w_\lambda=h &\mbox{on}\ \partial\Omega.
\end{array}
\right.
\ee

We prove the following result about the linearization of the problem \eqref{eq5}.
\begin{lem}\label{l1} For all $\lambda\in\R$ and $h\in\mathcal C^{2+\alpha}(\partial\Omega)$, the map $[-s_h,s_h]\ni s\longmapsto u_{s,\lambda}$ is lying in $\mathcal C^1([-s_h,s_h];\mathcal C^{2+\alpha}(\overline{\Omega}))$ and we have $\partial_su_{s,\lambda}|_{s=0}=w_\lambda$ with $w_\lambda$ the unique solution of \eqref{eq6}.\end{lem}
\begin{proof} Let us observe that applying the result of Proposition \ref{p1},  the unique solution of \eqref{eq1} with $g=\lambda+sh$, $s\in[-s_h,s_h]$, takes the form $\lambda+\psi(sh)$ with $\psi$ a $\mathcal C^1$ map  from $ B_{\epsilon_\lambda}$ to $\mathcal C^{2+\alpha}(\overline{\Omega})$. This clearly implies that $[-s_h,s_h]\ni s\longmapsto u_{s,\lambda}\in\mathcal C^1([-s_h,s_h];\mathcal C^{2+\alpha}(\overline{\Omega}))$. Moreover, using the fact that $\lambda$ is the unique solution of \eqref{eq1} with $g=\lambda$, we deduce that $ u_{s,\lambda}|_{s=0}=\lambda$. Therefore, we have 
$$\partial_s\left[-\sum_{i,j=1}^d \partial_{x_i} 
\left( a_{i,j}(x)\gamma(u_{\lambda,s}) \partial_{x_j} u_{\lambda,s} \right)\right]|_{s=0}=-\gamma(\lambda)\sum_{i,j=1}^d \partial_{x_i} 
\left( a_{i,j}(x) \partial_{x_j} \partial_su_{\lambda,s}|_{s=0} \right),$$
$$\partial_s[D(x,u_{\lambda,s})\cdot\nabla u_{\lambda,s}]|_{s=0}=D(x,\lambda)\cdot\nabla \partial_su_{\lambda,s}|_{s=0}$$
and it follows that 
$$\left\{
\begin{array}{ll}
-\gamma(\lambda)\sum_{i,j=1}^d \partial_{x_i} 
\left( a_{i,j}(x) \partial_{x_j} \partial_su_{\lambda,s}|_{s=0} \right)+D(x,\lambda)\cdot\nabla \partial_su_{\lambda,s}|_{s=0}=0  & \mbox{in}\ \Omega ,
\\
\partial_su_{\lambda,s}|_{s=0}=h &\mbox{on}\ \partial\Omega.
\end{array}
\right.$$
Then the uniqueness of the solution of this boundary value problem implies that $\partial_su_{\lambda,s}|_{s=0}=w_\lambda$.\end{proof}

Let us also define the partial DN map $\Lambda_{\gamma(\lambda),D(\cdot,\lambda)}:h\mapsto \gamma(\lambda)\partial_{\nu_a} w_\lambda|_{S}$ associated with problem \eqref{eq6}. For $h\in\mathcal C^{2+\alpha}(\partial\Omega)$,    $s\in[-s_h,s_h]$ we consider $u_{\lambda,s}$ the solution of \eqref{eq5} and $w_\lambda$ the solution of \eqref{eq6}.  Applying Lemma \ref{l1} and assuming that condition (i) is fulfilled, we get
$$\partial_s[\gamma(u_{\lambda,s})\partial_{\nu_a}u_{\lambda,s}]|_{s=0}=\partial_su_{\lambda,s}|_{s=0}\gamma'(u_{\lambda,s}|_{s=0})\partial_{\nu_a}u_{\lambda,s}|_{s=0}+\gamma(u_{\lambda,s}|_{s=0})\partial_{\nu_a}w_\lambda.$$
Recalling that $u_{\lambda,s}|_{s=0}=\lambda$, we get that $\partial_{\nu_a}u_{\lambda,s}|_{s=0}\equiv0$ and it follows that
$$\partial_s \mathcal N_{\gamma,F}(\lambda+sh)|_{s=0}=\gamma(\lambda)\partial_{\nu_a}w_\lambda|_{S}$$
which implies that
\bel{l1b} \partial_s \mathcal N_{\gamma,F}(\lambda+sh)|_{s=0}=\Lambda_{\gamma(\lambda),D(\cdot,\lambda)}h,\quad \lambda\in\R,\ h\in \mathcal C^{2+\alpha}(\partial\Omega).\ee
Using this identity, we deduce that \eqref{map} can be extended uniquely by density to a continuous linear map from $ H^{\frac{1}{2}}(\partial\Omega)$ to $H^{-\frac{1}{2}}(S)$. In order to prove that \eqref{mesure} is continuous, we need the following result.
\begin{lem}\label{l11}  The map $\R\ni\lambda\mapsto\Lambda_{\gamma(\lambda),D(\cdot,\lambda)}$ is lying in $\mathcal C(\R;\mathcal B(H^{\frac{1}{2}}(\partial\Omega),H^{-\frac{1}{2}}(S)))$.\end{lem}
\begin{proof}
From now on and in all this article, we denote by $\mathcal A$ the differential operator defined by
\bel{A}\mathcal Au(x)=-\sum_{i,j=1}^d \partial_{x_i} 
\left( a_{i,j}(x) \partial_{x_j} u(x)\right),\quad u\in H^1(\Omega),\ x\in\Omega.\ee
For any $h\in H^{\frac{1}{2}}(\partial\Omega)$ and any $\lambda\in\R$, we denote by $w_{\lambda,h}$ the solution of \eqref{eq6}. Since $\gamma\in \mathcal C^3(\R;(0,+\infty))$, the proof of the lemma will be completed if we show that 
$$\lim_{\delta\to 0}\ \ \underset{\norm{h}_{H^{\frac{1}{2}}(\partial\Omega)}\leq 1}{\sup}\norm{\partial_{\nu_a}w_{\lambda+\delta,h}-\partial_{\nu_a}w_{\lambda,h}}_{H^{-\frac{1}{2}}(\partial\Omega)}=0,\quad \lambda\in\R.$$
Using the fact that there exists $C>0$ depending only on $a$ and $\Omega$ such that
$$\norm{\partial_{\nu_a}w_{\lambda+\delta,h}-\partial_{\nu_a}w_{\lambda,h}}_{H^{-\frac{1}{2}}(\partial\Omega)}\leq C\left[\norm{w_{\lambda+\delta,h}-w_{\lambda,h}}_{H^1(\Omega)}+\norm{\mathcal Aw_{\lambda+\delta,h}-\mathcal Aw_{\lambda,h}}_{L^2(\Omega)}\right],$$
we are left with the task of proving that
\bel{l11a}\lim_{\delta\to 0}\ \ \underset{\norm{h}_{H^{\frac{1}{2}}(\partial\Omega)}\leq 1}{\sup}[\norm{w_{\lambda+\delta,h}-w_{\lambda,h}}_{H^1(\Omega)}+\norm{\mathcal Aw_{\lambda+\delta,h}-\mathcal Aw_{\lambda,h}}_{L^2(\Omega)}]=0,\quad \lambda\in\R.\ee
For this purpose, we fix $\lambda\in\R$, $\delta\in(-1,1)$, $h\in H^{\frac{1}{2}}(\partial\Omega)$ and we consider $w=w_{\lambda,h}-w_{\lambda+\delta,h}$. It is clear that $w$ solves the boundary value problem
\bel{eq66}
\left\{
\begin{array}{ll}
-\gamma(\lambda)\sum_{i,j=1}^d \partial_{x_i} 
\left( a_{i,j}(x) \partial_{x_j} w \right)+D(x,\lambda)\cdot\nabla w=G_{\delta,h}  & \mbox{in}\ \Omega ,
\\
w=0 &\mbox{on}\ \partial\Omega,
\end{array}
\right.
\ee
with $$G_{\delta,h}=[\gamma(\lambda+\delta)-\gamma(\lambda)]\mathcal A w_{\lambda+\delta,h}+[D(x,\lambda+\delta)-D(x,\lambda)]\cdot\nabla w_{\lambda+\delta,h}.$$
Using the fact that $\gamma\in \mathcal C^3(\R;(0,+\infty))$, $D\in \mathcal C^2(\overline{\Omega}\times\R;\R^n)$ and applying \cite[Theorem 8.3]{GT}, we deduce that there exists a constant $C>0$ depending only on $\lambda$, $a$, $\Omega$, $\gamma$ and $D$ such that
$$\norm{w_{s,h}}_{H^1(\Omega)}+\norm{\mathcal Aw_{s,h}}_{L^2(\Omega)}\leq C\norm{h}_{H^{\frac{1}{2}}(\partial\Omega)},\quad s\in[\lambda-1,\lambda+1].$$
Therefore, we find 
$$\norm{G_{\delta,h}}_{L^2(\Omega)}\leq C\left[|\gamma(\lambda+\delta)-\gamma(\lambda)|+\norm{D(\cdot,\lambda+\delta)-D(\cdot,\lambda)}_{L^\infty(\Omega)}\right]\norm{h}_{H^{\frac{1}{2}}(\partial\Omega)},$$
with $C>0$ independent of $\delta$ and $h$. Applying again \cite[Theorem 8.3]{GT}, we obtain
$$\norm{w}_{H^1(\Omega)}+\norm{\mathcal Aw}_{L^2(\Omega)}\leq C\left[|\gamma(\lambda+\delta)-\gamma(\lambda)|+\norm{D(\cdot,\lambda+\delta)-D(\cdot,\lambda)}_{L^\infty(\Omega)}\right]\norm{h}_{H^{\frac{1}{2}}(\partial\Omega)}$$
and, using the fact that $\gamma\in \mathcal C^3(\R;(0,+\infty))$, $D\in \mathcal C^2(\overline{\Omega}\times\R;\R^n)$, we get \eqref{l11a}.
\end{proof}

Combining \eqref{l1b} with Lemma \ref{l11},  we deduce that \eqref{mesure} is lying in $\mathcal C(\R)$.

Now let us consider the linearization of problem \eqref{eq1} when condition (ii) is fulfilled. For this purpose, we use the notation of Theorem \ref{t2} and we assume that the function $G$ in \eqref{non2}  satisfies 
$$|G(t)|+|G'(t)|+|G''(t)|\leq \kappa(|t|),\quad t\in\R,$$
with $\kappa\in\mathcal C([0,+\infty))$ a non-decreasing function. Then, according to \cite[Theorem 8.3, pp. 301]{LU} and \cite[Theorem 9.3, 9.4]{GT}, for $g=\lambda\chi+h$, with $\lambda\in[-R,R]$, $R>0$ and $h\in B_1$, the  problem \eqref{eq1} admits a unique solution $u\in \mathcal C^{2+\alpha}(\overline{\Omega})$ satisfying 
$\norm{u}_{L^\infty(\Omega)}\leq M(R)$ with $M$ a non-decreasing function of $\R_+$ depending only on $a$, $\Omega$, $\chi$ and $\kappa$.
We fix $h \in\mathcal C^{2+\alpha}(\partial\Omega)$, $\tau_h=\frac{1}{\norm{h}_{\mathcal C^{2+\alpha}(\partial\Omega)}+1}$, $s\in[-\tau_h,\tau_h]$, $\lambda\in\R$ and we consider the following boundary value problem 
\bel{eq7}
\left\{
\begin{array}{ll}
-\sum_{i,j=1}^d \partial_{x_i} 
\left( a_{i,j}(x) \partial_{x_j} v_{\lambda,s} \right)+G( v_{\lambda,s})=0  & \mbox{in}\ \Omega ,
\\
v_{\lambda,s}=\lambda\chi+sh &\mbox{on}\ \partial\Omega.
\end{array}
\right.
\ee
We consider also the solution of the following linear problem
\bel{eq8}
\left\{
\begin{array}{ll}
-\sum_{i,j=1}^d \partial_{x_i} 
\left( a_{i,j}(x) \partial_{x_j} p \right)+q_{\lambda,G}(x) p=0  & \mbox{in}\ \Omega ,
\\
p=h &\mbox{on}\ \partial\Omega,
\end{array}
\right.
\ee
where $q_{\lambda,G}(x)=G'(v_{\lambda,s}|_{s=0}(x))$, $x\in\Omega$. Then, following the argumentation of \cite{Is4}, we can show that.

\begin{lem}\label{l2} For all $h\in \mathcal C^{2+\alpha}(\partial\Omega)$, the map $[-\tau_h,\tau_h]\ni s\longmapsto v_{s,\lambda}$ is lying in $\mathcal C^1([-\tau_h,\tau_h];\mathcal C^{2+\alpha}(\overline{\Omega}))$ and we have $\partial_su_{s,\lambda}|_{s=0}=p$ with $p$ the unique solution of \eqref{eq8}.\end{lem}

Let us also define the partial DN map $\mathcal D_{q_{\lambda,G}}:h\mapsto \partial_{\nu_a} p_{|S}$ associated with problem \eqref{eq8}. Applying Lemma \ref{l2} and assuming that condition (ii) is fulfilled, we get
\bel{l2b} \partial_s \mathcal N_{\gamma,F}(\lambda\chi+sh)|_{s=0}=\mathcal D_{q_{\lambda,G}}h,\quad h\in  \mathcal C^{2+\alpha}(\partial\Omega).\ee
Let us also observe that according to the above discussions, we have $q_{\lambda,G}\in L^\infty(\Omega)$ and, for all $R>0$, we have 
\bel{l2c}\norm{q_{\lambda,G}}_{L^\infty(\Omega)}\leq \sup_{s\in[-M(R),M(R)]}|G'(s)|\leq \sup_{s\in[0,M(R)]}|\kappa(s)|\leq \kappa(M(R)),\quad \lambda\in [-R,R].\ee
Combining this estimate with the above argumentation, one can check that the map \eqref{mesure1} is lying in $\mathcal C(\R)$.

\section{Special solutions for the linearized problem}
In the spirit of the works \cite{A,EPS1,Is2}, the proofs of our main results are based on constructions of suitable class of singular solutions of the linearized problems \eqref{eq6} and \eqref{eq8}. In contrast to these last works, we need to consider a construction with variable second order coefficients and low order unknown coefficients. Moreover, we will localize the trace of such solutions on $\partial\Omega$ to the arbitrary subset $S$. For this purpose, we will introduce a new construction of these class of solutions designed for our inverse problem. We consider separately the  construction of these class of singular solutions for Theorem \ref{t1} and Theorem \ref{t2}.

\subsection{Special solutions for Theorem \ref{t1}}
We denote by $\Omega_\star$  a smooth open bounded set of $\R^n$ such that $\overline{\Omega}\subset \Omega_\star$ and we extend $a=(a_{i,j})_{1 \leq i,j \leq n}$ into a function of $\mathcal C^2(\overline{\Omega_\star};\R^{n^2})$ still denoted by $a=(a_{i,j})_{1 \leq i,j \leq n}$ and satisfying $a_{i,j}=a_{j,i}$ and
$$\sum_{i,j=1}^d a_{i,j}(x) \xi_i \xi_j \geq c |\xi|^2, \quad 
\mbox{for each $x \in \overline{\Omega_\star},\ \xi=(\xi_1,\ldots,\xi_n) \in \R^n$}.$$
From now on, for all $x\in\R^n$ and $r>0$, we set $B(x,r):=\{y\in\R^n:\ |y-x|<r\}$.

We introduce the function 
$$\rho(x,y):=\left(a(y)(x-y)\cdot (x-y)\right)^{\frac{1}{2}},\quad y\in \Omega_\star,\ x\in \R^n$$
and the function $f_n$ on $(0,+\infty)$ by
$$f_n(t):=\left\{\begin{aligned}-\frac{1}{2\pi}\ln(t)\quad &\textrm{for $n=2$},\\ \frac{t^{2-n}}{(n-2)d_n\sqrt{det(a(y))}}\quad&\textrm{for $n\geq3$},\end{aligned}\right.$$
with $d_n$ the area of the unit sphere.  We define
$$H(x,y):= f_n(\rho(x,y)),\quad y\in \Omega_\star,\ x\in\R^n\setminus\{y\}.$$
It is well known (see e.g. \cite[pp. 258]{Ka}) that $H$ satisfies the following estimates

\bel{est1} \begin{aligned}&|H(x,y)|\leq c_1\abs{\ln\left(\frac{|x-y|}{n^2\norm{a}_{L^\infty(\Omega)}}\right)},\quad y\in \Omega_\star,\ x\in\R^n\setminus\{y\},\  n=2,\\
&c_0|x-y|^{2-n}\leq H(x,y)\leq c_1|x-y|^{2-n},\quad y\in \Omega_\star,\ x\in\R^n\setminus\{y\},\ n\geq 3,\end{aligned}\ee
 where $c_j>0$, $j=0,1$, are  constants that depend only on $\Omega$ and $a$. In addition, for any $k=1,\ldots,n$, one can check that
\bel{est3} c_2|x-y|^{1-n}\leq |\nabla_x H(x,y)|\leq c_3|x-y|^{1-n},\quad y\in \Omega_\star,\ x\in\R^n\setminus\{y\},\ee
\bel{esti1} |\nabla_x\partial_{x_k} H(x,y)|\leq c_4|x-y|^{-n},\quad y\in \Omega_\star,\ x\in\R^n\setminus\{y\},\ee
where $c_j>0$, $j=2,\ldots,4,$ are  constants that depend only on $\Omega$ and $a$.
We fix $U=\{(x,x):\ x\in \Omega_\star\}$. It is well known (see e.g. \cite[Theorem 3]{Ka}) that there exists a function $P\in \mathcal C(\Omega_\star\times \Omega_\star\setminus U)$ taking the form $P=H+\mathcal R$ such that for all $y\in\Omega_\star$ we have $P(\cdot,y)\in \mathcal C^{2}(\Omega_\star\setminus\{y\})$ and
\bel{fond1}\sum_{i,j=1}^d \partial_{x_i} 
\left( a_{i,j}(x) \partial_{x_j} P(x,y) \right)=0,\quad x\in \Omega_\star\setminus\{y\},\ee
\bel{fond2}| \mathcal R(x,y)|\leq C|x-y|^{\frac{5}{2}-n},\quad |\nabla_x\mathcal R(x,y)|\leq C|x-y|^{\frac{3}{2}-n},\quad y\in \Omega_\star,\ x\in\R^n\setminus\{y\},\ee
with $C>0$ depending only on $\Omega$ and $a$.
We fix $x_0\in S$ and using boundary normal coordinates (see e.g. \cite[Theorem 2.12]{KKL}) we fix $\delta'>0$ sufficiently small such that for all $\tau\in(0,\delta')$ there exists a unique $y_\tau\in \Omega_\star\setminus \overline{\Omega}$ such that dist$(y_\tau,\partial\Omega)=|y_\tau-x_0|=\tau$. We fix $S'$ an open subset of $S$ containing $x_0$ and we consider $\Omega'$ an open bounded set of $\R^n$ with $\mathcal C^2$ boundary such that $\overline{\Omega'}\subset \Omega_\star$, $\Omega\subset \Omega'$ and $(\partial\Omega\setminus S')\subset \partial\Omega'$, $x_0\in\Omega'$ (see Figure 1). Fixing $\delta=\min($dist$(x_0,\partial\Omega')/3,1,\delta')>0$,  for all $\tau\in(0,\delta)$ we have dist$(y_\tau,\partial\Omega')\geq\delta$.

\begin{figure}[!ht]
  \centering
  \includegraphics[width=0.6\textwidth]{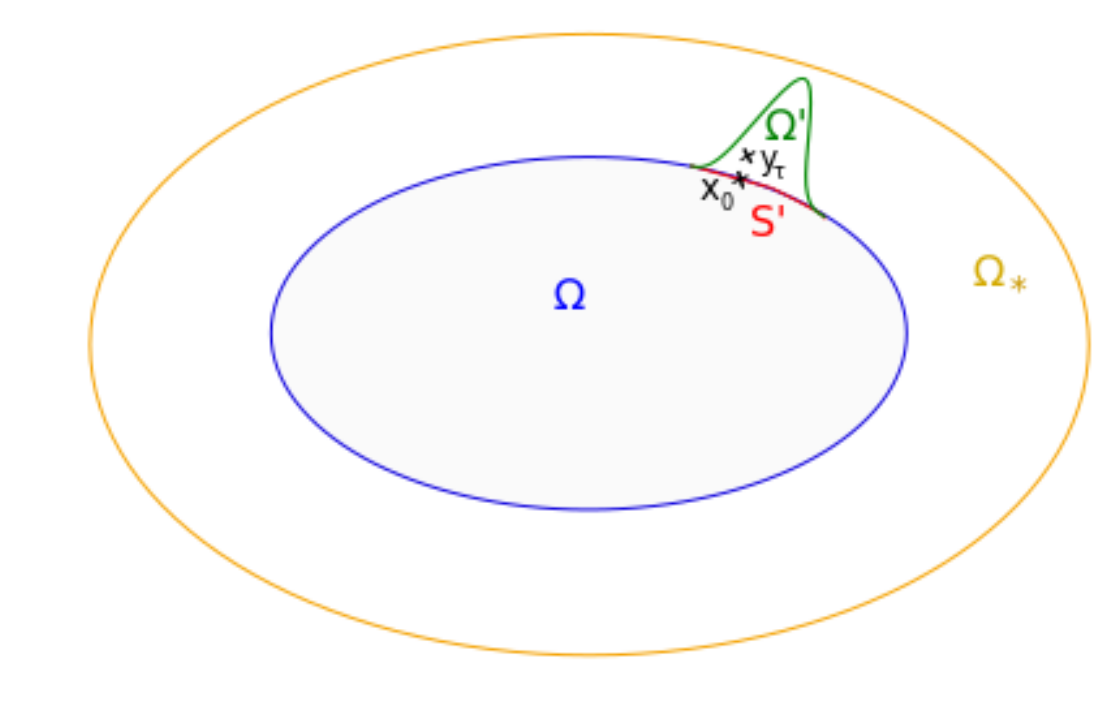}
  \caption{The sets $\Omega$, $\Omega_\star$, $\Omega'$ and the points $x_0$, $y_\tau$.  }
  \label{fig1}
\end{figure}

Then, we consider $w_\tau^0(x)$ the solution of the following boundary value problem
$$\left\{
\begin{array}{ll}
-\sum_{i,j=1}^d \partial_{x_i} 
\left( a_{i,j}(x) \partial_{x_j} w_\tau^0(x) \right)=0  & \mbox{in}\ \Omega' ,
\\
w_\tau^0=P(\cdot,y_\tau) &\mbox{on}\ \partial\Omega',
\end{array}
\right.$$
and for all $\tau\in (0,\delta)$ we deduce that $w_\tau^0\in H^2(\Omega)$ and, applying estimates \eqref{est1}-\eqref{est3} and \eqref{fond2}, we get 
\bel{est4}\norm{w_\tau^0}_{H^1(\Omega')}\leq C\norm{P(\cdot,y_\tau)}_{H^{\frac{1}{2}}(\partial\Omega')}\leq C,\ee
with $C$ depending only on $\Omega'$ and $a$. We set $g_\tau^1\in H^{\frac{3}{2}}(\partial\Omega)$ defined by
\bel{g1}g_\tau^1(x)=P(x,y_\tau)-w_\tau^0(x),\quad x\in\partial\Omega,\ \tau\in(0,\delta).\ee
Note that here we have
$$g_\tau^1(x)=P(x,y_\tau)-w_\tau^0(x)=P(x,y_\tau)-P(x,y_\tau)=0,\quad x\in\partial\Omega\setminus S',$$
which implies that supp$(g_\tau)\subset S$.
We fix $B\in W^{1,\infty}(\Omega;\R^n)$  such that $\nabla\cdot B\leq0$
and, for $s\in(0,+\infty)$,  $\tau\in(0,\delta)$, we consider $w_{s,\tau}\in H^2(\Omega)$  the solution of the following boundary value problem
\bel{eqq1}\left\{
\begin{array}{ll}
-s\sum_{i,j=1}^d \partial_{x_i} 
\left( a_{i,j}(x) \partial_{x_j} w_{s,\tau}(x) \right)+B(x)\cdot\nabla w_{s,\tau} =0  & \mbox{in}\ \Omega ,
\\
w_{s,\tau}=g_\tau^1 &\mbox{on}\ \partial\Omega.
\end{array}
\right.\ee
For $s\in(0,+\infty)$,  $\tau\in(0,\delta)$, we consider also $w_{s,\tau}^*\in H^2(\Omega)$  the solution of the adjoint problem
\bel{eqq2}\left\{
\begin{array}{ll}
-s\sum_{i,j=1}^d \partial_{x_i} 
\left( a_{i,j}(x) \partial_{x_j} w_{s,\tau}^*(x) \right)-\nabla\cdot(B(x) w_{s,\tau}^*) =0  & \mbox{in}\ \Omega ,
\\
w_{s,\tau}^*=g_\tau^1 &\mbox{on}\ \partial\Omega.
\end{array}
\right.\ee
Note that the well-posedness of \eqref{eqq1}-\eqref{eqq2} can be deduced from \cite[Theorem 8.3]{GT}.
We can prove the following result.
\begin{prop}\label{p2} For all $s\in(0,+\infty)$ and all $\tau\in (0,\delta)$, we have 
\bel{p2a}w_{s,\tau}(x)=H(x,y_\tau) +z_{s,\tau}(x),\quad w_{s,\tau}^*(x)=H(x,y_\tau) +z_{s,\tau}^*(x),\ee
where $z_{s,\tau},z_{s,\tau}^*\in H^2(\Omega)$ satisfy the estimate
\bel{p2b}\norm{z_{s,\tau}}_{H^1(\Omega)}+\norm{z_{s,\tau}^*}_{H^1(\Omega)}\leq C\max(1,\tau^{\frac{3}{2}-\frac{n}{2}}),\quad s\in(0,+\infty),\ \tau\in (0,\delta)\ee
with $C>0$ depending only on $a$, $s$, $\Omega$, $B$.\end{prop}

\begin{proof} We will only show the result for $w_{s,\tau}$ the result for $w_{s,\tau}^*$ being treated in a similar fashion. Let us first observe that we can divide $z_{s,\tau}$ into two terms $z_{s,\tau}(x)=\mathcal R(x,y_\tau)+h_{s,\tau}(x)$ where $h_{s,\tau}$ solves the problem
\bel{eqq3}\left\{
\begin{array}{ll}
-s\sum_{i,j=1}^d \partial_{x_i} 
\left( a_{i,j}(x) \partial_{x_j} h_{s,\tau}(x) \right)+B(x)\cdot\nabla h_{s,\tau}(x) =-B(x)\cdot\nabla (P(x,y_\tau)+w_\tau^0(x))  & \ x\in\Omega ,
\\
h_{s,\tau}=0 &\mbox{on}\ \partial\Omega.
\end{array}
\right.\ee
According to \cite[Theorem 8.1]{GT}, there exists  $C>0$ depending only on $a$, $s$, $\Omega$, $B$ such that
\bel{p2c}\norm{h_{s,\tau}}_{H^1(\Omega)}\leq C\norm{B\cdot\nabla (P(\cdot,y_\tau)+w_\tau^0)}_{H^{-1}(\Omega)}\leq C\norm{B}_{W^{1,\infty}(\Omega)}\norm{P(\cdot,y_\tau)+w_\tau^0}_{L^2(\Omega)}.\ee
Consider $R>0$ such that $\Omega\subset B(y_\tau,R)$. Using the fact that $B(y_\tau,\tau)\cap\Omega=\emptyset$, we deduce that $\Omega\subset B(y_\tau,R)\setminus  B(y_\tau,\tau)$. Therefore, applying \eqref{est1}, \eqref{fond2} and \eqref{est4}, we obtain 
$$\norm{P(\cdot,y_\tau)+w_\tau^0}_{L^2(\Omega)}^2\leq C\left(1+\int_{B(y_\tau,R)\setminus  B(y_\tau,\tau)}|x-y_\tau|^{4-2n}dx\right)\leq C(1+\tau^{4-n}),\quad \tau\in (0,\delta),$$
with $C>0$ depending only on $a$, $s$, $\Omega$, $B$. Combining this last estimate with \eqref{p2c}, we obtain
$$\norm{h_{s,\tau}}_{H^1(\Omega)}\leq C\max(1,\tau^{2-\frac{n}{2}}),\quad \tau\in (0,\delta).$$
In the same way, estimates \eqref{fond2} implies that
$$\norm{\mathcal R(\cdot,y_\tau)}_{H^1(\Omega)}\leq C\max(1,\tau^{\frac{3}{2}-\frac{n}{2}}),\quad \tau\in (0,\delta).$$
which clearly implies \eqref{p2b}.\end{proof}

Let us also observe that, following the argumentation of Proposition \ref{p2}, we can prove that the estimates \eqref{est1}-\eqref{esti1} and \eqref{fond2} imply 
\bel{g3}\norm{g_\tau^1}_{H^{\frac{1}{2}}(\partial\Omega)}\leq C\max(\tau^{1-\frac{n}{2}},|\ln(\tau)|^{\frac{1}{2}}),\quad \tau\in(0,\delta).\ee

\subsection{Special solutions for Theorem \ref{t2}}
In this subsection assume that  $n\geq3$ and that  condition \eqref{t2b} is fulfilled. Under such assumption, we fix
$$P(x,y)=H(x,y):= \frac{|x-y|^{2-n}}{(n-2)d_n},\quad y\in \R^n,\ x\in\R^n\setminus\{y\}.$$
and we recall that condition \eqref{est1}-\eqref{esti1} are fulfilled.  Note also that in such a context, for any $k=1,\ldots,n$, we have 
$$\Delta_x\partial_{x_k}H(x,y)=\partial_{x_k}\Delta_xH(x,y)=0,\quad y\in\R^n,\ x\in\R^n\setminus\{y\}.$$
We give the same definition as the preceding subsection to the sets $\Omega_\star$, $\Omega'$, $S'$ and the points $x_0$ and $y_\tau$, $\tau\in(0,\delta)$.
Using the fact that $\partial_{x_k}H(\cdot,y_\tau)\in \mathcal C^\infty(\overline{\Omega'}\setminus \{y_\tau\})$, $k=1,\ldots,n$, we can define
 $v_{k,\tau}\in H^1(\Omega')$ the solution of the following boundary value problem
$$\left\{
\begin{array}{ll}
-\Delta  v_{k,\tau}  =0  & \mbox{in}\ \Omega' ,
\\
v_{k,\tau}=\partial_{x_k}H(\cdot,y_\tau) &\mbox{on}\ \partial\Omega'
\end{array}
\right.$$
and we deduce again that
\bel{est6}\norm{v_{k,\tau}}_{H^1(\Omega')}\leq C\norm{\partial_{x_k}H(\cdot,y_\tau)}_{H^{\frac{1}{2}}(\partial\Omega')}\leq C,\ee
with $C$ depending only on $\Omega'$. We set $g_{k,\tau}^2\in H^{\frac{1}{2}}(\partial\Omega)$ defined by
\bel{g2}g_{k,\tau}^2(x)=\partial_{x_k}H(x,y_\tau)-v_{k,\tau}(x),\quad x\in\partial\Omega,\ \tau\in(0,\delta)\ee
and we notice again that supp$(g_{k,\tau}^2)\subset S$. For $\tau\in(0,\delta)$, we consider $w_{k,\tau}\in H^1(\Omega)$ the solution of the following boundary value problem
\bel{eqq4}\left\{
\begin{array}{ll}
-\Delta w_{k,\tau}(x) +qw_{k,\tau}(x) =0  & \mbox{in}\ \Omega ,
\\
w_{k,\tau}=g_{k,\tau}^2 &\mbox{on}\ \partial\Omega,
\end{array}\right.\ee
where $q$ is a non-negative function satisfying
$$\norm{q}_{L^\infty(\Omega)}\leq M$$
for some $M>0$. This solution satisfies the following properties.

\begin{prop}\label{p3} Let condition \eqref{t2b} be fulfilled. Then for all $\tau\in (0,\delta)$, the solution $w_{k,\tau}$ of \eqref{eqq4} takes the form
\bel{p3a}w_{k,\tau}(x)=\partial_{x_k}H(x,y_\tau) +J_{k,\tau}(x),\quad x\in\Omega,\ee
where $J_{k,\tau}\in H^1(\Omega)$ and satisfies the estimate
\bel{p3b}\norm{J_{k,\tau}}_{H^1(\Omega)}\leq C\max(1,\tau^{2-\frac{n}{2}}),\quad  \tau\in (0,\delta),\ee
with $C>0$ depending only on   $\Omega$ and $M$.\end{prop}
\begin{proof}In all this proof we denote by $C>0$ a positive constant depending on $\Omega$ and $M$ that may change from line to line.
 Note first that $h_{k,\tau}=J_{k,\tau}-v_{k,\tau}$  solves the problem
\bel{eqq5}\left\{
\begin{array}{ll}
-\Delta h_{k,\tau}+q(x)h_{k,\tau} =-q(\partial_{x_k}H(\cdot,y_\tau)+v_{k,\tau})  & \mbox{in}\ \Omega ,
\\
h_{k,\tau}=0 &\mbox{on}\ \partial\Omega.
\end{array}
\right.\ee
Applying the Sobolev embedding theorem, we deduce that  $H^1_0(\Omega)$ embedded into  $L^{\frac{2n}{n-2}}(\Omega)$. Therefore, by duality we deduce that  the space $L^{\frac{2n}{n+2}}(\Omega)$ embedded into $H^{-1}(\Omega)$. Thus, the solution of \eqref{eqq5} satisfies the estimate
$$\begin{aligned}\norm{h_{k,\tau}}_{H^1(\Omega)}\leq C\norm{q(\partial_{x_k}H(\cdot,y_\tau)+v_{k,\tau}) }_{H^{-1}(\Omega)}&\leq C\norm{q(\partial_{x_k}H(\cdot,y_\tau)+v_{k,\tau}) }_{L^{\frac{2n}{n+2}}(\Omega)}\\
&\leq C\norm{q}_{L^\infty(\Omega)}\norm{(\partial_{x_k}H(\cdot,y_\tau)+v_{k,\tau})}_{L^{\frac{2n}{n+2}}(\Omega)}.\end{aligned}$$
On the other hand, repeating the arguments used in Proposition \ref{p1} and applying estimate \eqref{est3} we can prove that
$$\begin{aligned}\norm{\partial_{x_k}H(\cdot,y_\tau)+v_{k,\tau}}_{L^{\frac{2n}{n+2}}(\Omega)}&\leq C\left(\int_\Omega|x-y_\tau|^{{\frac{2n(1-n)}{n+2}}}dx+ 1\right)^{\frac{n+2}{2n}}\\
&\leq C\left(\int_{B(y_\tau,R)\setminus  B(y_\tau,\tau)}|x-y_\tau|^{{\frac{2n(1-n)}{n+2}}}dx+1\right)^{\frac{n+2}{2n}},\\
&\leq C\left(1+\tau^{\frac{2n(1-n)}{n+2}+n}+1\right)^{\frac{n+2}{2n}}\leq C(1+\tau^{2-\frac{n}{2}})\quad \tau\in (0,\delta),\end{aligned}$$
which implies that 
$$\norm{h_{k,\tau}}_{H^1(\Omega)}\leq C\max(1,\tau^{2-\frac{n}{2}}),\quad \tau\in (0,\delta).$$
Combining this with \eqref{est6} we deduce \eqref{p3b}.\end{proof}

Following the argumentation of Proposition \ref{p3}, we can prove that the estimates \eqref{est1}-\eqref{esti1}  imply 
\bel{g4}\norm{g_{k,\tau}^2}_{H^{\frac{1}{2}}(\partial\Omega)}\leq C\tau^{-\frac{n}{2}},\quad \tau\in(0,\delta),\ee
with $C>0$ depending only on $\Omega$.

\section{Proof of the main results}

\subsection{Proof of Theorem \ref{t1}}
We use here the notation of Section 3.
Let us first observe that in view of \eqref{l1b}, fixing $\mathcal N=\mathcal N_{\gamma_1,F_1}-\mathcal N_{\gamma_2,F_2}$, we have
$$\partial_s \mathcal N(\lambda+sh)|_{s=0}=\Lambda_{\gamma_1(\lambda),D_1(\cdot,\lambda)}h-\Lambda_{\gamma_2(\lambda),D_2(\cdot,\lambda)}h,\quad h\in \mathcal H_S,$$
where we recall that  $\Lambda_{\gamma(\lambda),D(\cdot,\lambda)}$ denotes the partial DN map associated with problem \eqref{eq6}.
Therefore, fixing $H^{\frac{1}{2}}_0(S):=\{f\in H^{\frac{1}{2}}(\partial\Omega):\ \textrm{supp}(f)\subset S\}$ and $\mathcal C^{2+\alpha}_0(S):=\{f\in \mathcal C^{2+\alpha}(\partial\Omega):\ \textrm{supp}(f)\subset S\}$, denoting by $\Lambda^j_{S,\lambda}$, $j=1,2$, the restriction of $\Lambda_{\gamma_j(\lambda),D_j(\cdot,\lambda)}$ to $H^{\frac{1}{2}}_0(S)$ and applying the density of $\mathcal C^{2+\alpha}_0(S)$ in $H^{\frac{1}{2}}_0(S)$, we obtain
$$\begin{aligned}\norm{\Lambda^1_{S,\lambda}-\Lambda^2_{S,\lambda}}_{\mathcal B(H^{\frac{1}{2}}_0(S);H^{-\frac{1}{2}}(S))}&= \sup_{h\in\mathcal H_S}\norm{\Lambda^1_{S,\lambda}h-\Lambda^2_{S,\lambda}h}_{H^{-\frac{1}{2}}(S)}\\
&=\sup_{h\in\mathcal H_S}\norm{\Lambda^1_{S,\lambda}h-\Lambda^2_{S,\lambda}h}_{H^{-\frac{1}{2}}(S)}\\
&=\sup_{h\in\mathcal H_S}\limsup_{\epsilon\to0}\norm{ \frac{\mathcal N(\lambda+\epsilon h)-\mathcal N(\lambda)}{\epsilon}}_{H^{-\frac{1}{2}}(S)}.\end{aligned}$$
It follows that, for all $R>0$, we have
\bel{t1d}\sup_{\lambda\in[-R,R]}\norm{\Lambda^1_{S,\lambda}-\Lambda^2_{S,\lambda}}_{\mathcal B(H^{\frac{1}{2}}_0(S);H^{-\frac{1}{2}}(S))}=\sup_{\lambda\in[-R,R]}\sup_{h\in\mathcal H_S}\limsup_{\epsilon\to0}\frac{\norm{\mathcal N(\lambda+\epsilon h)-\mathcal N(\lambda)}_{H^{-\frac{1}{2}}(S)}}{\epsilon}\ee
and the proof of Theorem \ref{t1} will be completed if we show that there exists a constant $C>0$, depending on $a$, $\Omega$, such that the following estimate 
\bel{t1e}\norm{\gamma_1-\gamma_2}_{L^\infty(-R,R)}\leq C\sup_{\lambda\in[-R,R]}\norm{\Lambda^1_{S,\lambda}-\Lambda^2_{S,\lambda}}_{\mathcal B(H^{\frac{1}{2}}_0(S);H^{-\frac{1}{2}}(S))},\quad R>0,\ee 
holds true. Fixing $R>0$, we define $\lambda_R\in [-R,R]$ such that for $\gamma=\gamma_1-\gamma_2\in\mathcal C^3(\R)$ we have
$$\norm{\gamma_1-\gamma_2}_{L^\infty(-R,R)}=|\gamma(\lambda_R)|$$
and without loss of generality we assume that $\gamma(\lambda_R)\geq0$. In light of condition \eqref{t11a}, for $\tau\in(0,\delta)$ and $j=1,2$, we can consider $w_{j,\tau}$ (resp. $w_{j,\tau}^*$) the solution of \eqref{eqq1} (resp. \eqref{eqq2}) with $s=\gamma_j(\lambda_R)$, $B=D_j(\cdot,\lambda_R)$. We fix $D=D_1(\cdot,\lambda_R)-D_2(\cdot,\lambda_R)$ and we recall that $\mathcal A$ is defined by \eqref{A}.
Fixing $w_\tau=w_{2,\tau}-w_{1,\tau}$, we deduce that $w_\tau$ solves the problem
\bel{t1x}\left\{
\begin{array}{ll}
\gamma_2(\lambda_R)\mathcal A w_\tau+D_2(x,\lambda_R)\cdot\nabla w_\tau =K_\tau  & \mbox{in}\ \Omega ,
\\
w_\tau=0 &\mbox{on}\ \partial\Omega,
\end{array}
\right.\ee
with
$K_\tau=\gamma(\lambda_R)\mathcal Aw_{1,\tau}+D\cdot\nabla w_{1,\tau}$.
Multiplying \eqref{t1x} by $w_{2,\tau}^*$ and integrating by parts, we obtain
$$-\gamma_2(\lambda_R)\int_{\partial\Omega}\partial_{\nu_a}w_\tau w_{2,\tau}^*d\sigma(x)=\int_\Omega K_\tau w_{2,\tau}^*dx.$$
Integrating again by parts, we find
$$\begin{aligned}\int_\Omega K_\tau w_{2,\tau}^*dx=&-\gamma(\lambda_R)\int_{\partial\Omega}\partial_{\nu_a}w_{1,\tau} w_{2,\tau}^*d\sigma(x)+\gamma(\lambda_R)\int_\Omega\sum_{i,j=1}^d 
 a_{i,j}(x) \partial_{x_j} w_{1,\tau}(x)\partial_{x_i}w_{2,\tau}^*(x)dx\\
&+\int_\Omega D\cdot\nabla w_{1,\tau}w_{2,\tau}^*dx.\end{aligned}$$
Combining this with the fact that
$$-\gamma_2(\lambda_R)\partial_{\nu_a}w_\tau(x)+\gamma(\lambda_R)\partial_{\nu_a}w_{1,\tau}(x)=\gamma_1(\lambda_R)\partial_{\nu_a}w_{1,\tau}(x)-\gamma_2(\lambda_R)\partial_{\nu_a}w_{2,\tau}(x),\quad x\in\partial\Omega$$
and using the fact that $g_\tau^1\in H^{\frac{1}{2}}_0(S)$, we obtain the identity
\bel{t1f} \begin{aligned}&\gamma(\lambda_R)\int_\Omega\sum_{i,j=1}^d 
 a_{i,j}(x) \partial_{x_j} w_{1,\tau}(x)\partial_{x_i}w_{2,\tau}^*(x)dx+\int_\Omega (D_1(\cdot,\lambda_R)-D_2(\cdot,\lambda_R))\cdot\nabla w_{1,\tau}(x)w_{2,\tau}^*(x)dx\\
&=\left\langle (\Lambda^1_{S,\lambda_R}-\Lambda^2_{S,\lambda_R})g_\tau^1,g_\tau^1\right\rangle_{H^{-\frac{1}{2}}(S),H^{\frac{1}{2}}_0(S)}.\end{aligned}\ee
In view of Proposition \ref{p2}, we can decompose $w_{1,\tau}$ and $w_{2,\tau}^*$ into two terms
\bel{t1g}w_{1,\tau}(x)=H(x,y_\tau) +z_{1,\tau}(x),\quad w_{2,\tau}^*(x)=H(x,y_\tau) +z_{2,\tau}^*(x),\quad x\in\Omega,\ee
with 
\bel{t1h}\norm{z_{1,\tau}}_{H^1(\Omega)}+\norm{z_{2,\tau}^*}_{H^1(\Omega)}\leq C_1\max(1,\tau^{\frac{3-n}{2}}),\quad \tau\in(0,\delta),\ee
with $C_1>0$ independent of $\tau$. In the same way, applying \eqref{est1}-\eqref{est3} and repeating the argumentation of Proposition \ref{p2}, we get
\bel{t1i} \tau\norm{H(\cdot,y_\tau)}_{H^1(\Omega)}+\norm{H(\cdot,y_\tau)}_{L^2(\Omega)}\leq C_2\max(1,\tau^{2-\frac{n}{2}}),\quad \tau\in(0,\delta),\ee
with $C_2>0$ independent of $\tau$.
Applying estimates \eqref{t1h}-\eqref{t1i}, we obtain
$$\tau \norm{w_{1,\tau}}_{H^1(\Omega)}+\norm{w_{2,\tau}^*}_{L^2(\Omega)}\leq C_3\max(1,\tau^{2-\frac{n}{2}}),\quad \tau\in(0,\delta),$$
with $C_3>0$ independent of $\tau$ and it follows that
$$\begin{aligned}&\abs{\int_\Omega (D_1(\cdot,\lambda_R)-D_2(\cdot,\lambda_R))\cdot\nabla w_{1,\tau}(x)w_{2,\tau}^*(x)dx}\\
&\leq \norm{D_1(\cdot,\lambda_R)-D_2(\cdot,\lambda_R)}_{L^\infty(\Omega)}\norm{w_{1,\tau}}_{H^1(\Omega)}\norm{w_{2,\tau}^*}_{L^2(\Omega)}\\
\ &\leq C_4\max(1,\tau^{3-n}),\quad \tau\in(0,\delta),\end{aligned}$$
with $C_4>0$ a constant independent of $\tau$ which depends on $D_1$, $D_2$, $R$, $\gamma_1$, $\gamma_2$. In the same way, applying \eqref{t1h}-\eqref{t1i} and \eqref{ell}, we get
$$\begin{aligned}&\abs{\int_\Omega\sum_{i,j=1}^d 
 a_{i,j}(x) \partial_{x_j} w_{1,\tau}(x)\partial_{x_i}w_{2,\tau}^*(x)dx}\\
&\geq \int_\Omega\sum_{i,j=1}^d 
 a_{i,j}(x) \partial_{x_j} H(x,y_\tau)\partial_{x_i}H(x,y_\tau)dx -C_5\max(1,\tau^{\frac{5}{2}-n})\\
&\geq c\int_\Omega |\nabla_x H(x,y_\tau)|^2dx-C_5\max(1,\tau^{\frac{5}{2}-n}),\quad \tau\in(0,\delta),\end{aligned}$$
with $C_5>0$ a constant independent of $\tau$ which depends on $D_1$, $D_2$, $R$, $\gamma_1$, $\gamma_2$. In addition, applying \eqref{g3}, for all $\tau\in(0,\delta)$, we obtain 
$$\begin{aligned}|\left\langle (\Lambda^1_{S,\lambda_R}-\Lambda^2_{S,\lambda_R})g_\tau^1,g_\tau^1\right\rangle_{H^{-\frac{1}{2}}(S),H^{\frac{1}{2}}_0(S)}|&\leq \norm{\Lambda^1_{S,\lambda_R}-\Lambda^2_{S,\lambda_R}}_{\mathcal B(H^{\frac{1}{2}}_0(S),H^{-\frac{1}{2}}(S))}\norm{g_\tau^1}_{H^{\frac{1}{2}}(\partial\Omega)}^2\\
&\leq C_*\norm{\Lambda^1_{S,\lambda_R}-\Lambda^2_{S,\lambda_R}}_{\mathcal B(H^{\frac{1}{2}}_0(S),H^{-\frac{1}{2}}(S))}\max(\tau^{2-n},|\ln(\tau)|),\end{aligned}$$
with $C_*>0$ depending only on $\Omega$ and $a$. Combining all these estimates  with \eqref{t1f}, for all $\tau\in(0,\delta)$, we have
\bel{t1j} c\gamma(\lambda_R)\int_\Omega |\nabla_x H(x,y_\tau)|^2dx-\gamma(\lambda_R)C_5\tau^{\frac{5}{2}-n}-C_4\tau^{3-n}\leq C_*\norm{\Lambda^1_{S,\lambda_R}-\Lambda^2_{S,\lambda_R}}_{\mathcal B(H^{\frac{1}{2}}_0(S),H^{-\frac{1}{2}}(S))}\tau^{2-n}.\ee
On the other hand, applying \eqref{est3}, one can find $\delta_1\in(0,\delta)$ such that for all $\tau\in (0,\delta_1)$, we have
$$\begin{aligned}\int_\Omega |\nabla_x H(x,y_\tau)|^2dx&\geq \int_{B(y_\tau,4\tau)\cap\Omega} |\nabla_x H(x,y_\tau)|^2dx\\
&\geq c_2\int_{B(y_\tau,4\tau)\cap\Omega}|x-y_\tau|^{2-2n}\\
&\geq c'\max(\tau^{2-n},|\ln(\tau)|),\quad \tau\in (0,\delta_1),\end{aligned}$$
with $c'>0$ a constant depending on $a$ and $\Omega$. Combining this with \eqref{t1j}, for all $\tau\in (0,\delta_1)$, we obtain
$$\begin{aligned}&cc'\gamma(\lambda_R)\max(\tau^{2-n},|\ln(\tau)|)-\gamma(\lambda_R)C_5\max(1,\tau^{\frac{5}{2}-n})-C_4\max(1,\tau^{3-n})\\
&\leq C_*\norm{\Lambda^1_{S,\lambda_R}-\Lambda^2_{S,\lambda_R}}_{\mathcal B(H^{\frac{1}{2}}_0(S),H^{-\frac{1}{2}}(S))}\max(\tau^{2-n},|\ln(\tau)|).\end{aligned}$$
Dividing both side of this inequality by $\max(\tau^{2-n},|\ln(\tau)|)$ and sending $\tau\to0$, we obtain
$$cc'\gamma(\lambda_R)\leq C_*\norm{\Lambda^1_{S,\lambda_R}-\Lambda^2_{S,\lambda_R}}_{\mathcal B(H^{\frac{1}{2}}_0(S),H^{-\frac{1}{2}}(S))}.$$
Since $c$, $c'$ and $C_*$ are constants depending only on $\Omega$ and $a$, we deduce that there exists $C>0$ depending only on $\Omega$ and $a$ such that
$$\begin{aligned}\norm{\gamma_1-\gamma_2}_{L^\infty(-R,R)}=\gamma(\lambda_R)&\leq C\norm{\Lambda^1_{S,\lambda_R}-\Lambda^2_{S,\lambda_R}}_{\mathcal B(H^{\frac{1}{2}}_0(S),H^{-\frac{1}{2}}(S))}\\
&\leq C\sup_{\lambda\in[-R,R]}\norm{\Lambda^1_{S,\lambda}-\Lambda^2_{S,\lambda}}_{\mathcal B(H^{\frac{1}{2}}_0(S),H^{-\frac{1}{2}}(S))}.\end{aligned}$$
Combining this with \eqref{t1d}, we obtain \eqref{t1c}. This completes the proof of the theorem.

\subsection{Proof of Theorem \ref{t2}}
Fix $R>0$ and consider, for all $\lambda\in[-R,R]$, $q_{j,\lambda}(x)=G_j'(v_{j,\lambda}(x))$, $j=1,2$, with $v_{j,\lambda}$ solving the problem
$$\left\{
\begin{array}{ll}
-\Delta v_{j,\lambda} +G_j( v_{j,\lambda})=0  & \mbox{in}\ \Omega ,
\\
v_{j,\lambda}=\lambda\chi &\mbox{on}\ \partial\Omega.
\end{array}
\right.$$
We denote by $\mathcal D^j_{S,\lambda}$, $j=1,2$, the restriction of $\mathcal D_{q_{j,\lambda}}$ to $H^{\frac{1}{2}}_0(S)$, where we recall that $\mathcal D_{q_{\lambda,G}}$ is the DN map associated with problem \eqref{eq8}. Applying Lemma \ref{l2} and repeating the argumentation of Theorem \ref{t1}, we can prove that
$$\sup_{\lambda\in[-R,R]}\norm{\mathcal D^1_{S,\lambda}-\mathcal D^2_{S,\lambda}}_{\mathcal B(H^{\frac{1}{2}}_0(S);H^{-\frac{1}{2}}(S))}=\sup_{\lambda\in[-R,R]}\sup_{h\in\mathcal H_S}\limsup_{\epsilon\to0}\frac{\norm{\mathcal N(\lambda+\epsilon h)-\mathcal N(\lambda)}_{H^{-\frac{1}{2}}(S)}}{\epsilon}$$
and the proof Theorem \ref{t2} will be completed if we show that for all $R>0$ there exists a constant $C>0$, depending on  $\Omega$, $\kappa$, $\chi$ and $R$ such that the following estimate 
\bel{t2e}\norm{G_1-G_2}_{L^\infty(-R,R)}\leq C\left(\sup_{\lambda\in[-R,R]}\norm{\mathcal D^1_{S,\lambda}-\mathcal D^2_{S,\lambda}}_{\mathcal B(H^{\frac{1}{2}}_0(S);H^{-\frac{1}{2}}(S))}\right)^{\frac{1}{3}}\ee 
holds true. In a similar way to Theorem \ref{t1}, we fix $G=G_1-G_2$ and we consider $\lambda_R\in[-R,R]$ such that $\norm{G_1'-G_2'}_{L^\infty(-R,R)}=|G'(\lambda_R)|$. Without loss of generality we assume that $G'(\lambda_R)\geq0$. Fix $\tau\in (0,\delta)$, $k=1,\ldots,n$ and consider the solution $w_{j,k,\tau}$ of \eqref{eqq4} with $q=q_{j,\lambda_R}$. Note that here, in view of \eqref{t2a} and \eqref{l2c}, there exists $M>0$ depending only on $\Omega$, $\kappa$, $\chi$  and $R$ such that $\norm{q_{j,\lambda_R}}_{L^\infty(\Omega)}\leq M$.
Therefore, applying Proposition \ref{p3}, we can decompose $w_{j,k,\tau}$ into two terms $ w_{j,k,\tau}=\partial_{x_k}H(\cdot,y_\tau)+J_{j,k,\tau}$ with 
\bel{t2f}\norm{J_{j,k,\tau}}_{H^1(\Omega)}\leq C\max(1,\tau^{2-\frac{n}{2}}),\quad  \tau\in (0,\delta),\ee
where $C>0$ is a constant depending only on $\Omega$, $\chi$, $\kappa$ and $R$. In a similar way to Theorem \ref{t1}, integrating by parts, we obtain the identity
$$ \begin{aligned}&\int_\Omega (q_{1,\lambda_R}(x)-q_{2,\lambda_R}(x)) w_{1,k,\tau}(x)w_{2,k,\tau}(x)dx\\
&=\left\langle (\mathcal D^1_{S,\lambda_R}-\mathcal D^2_{S,\lambda_R})g_{k,\tau}^2,g_{k,\tau}^2\right\rangle_{H^{-\frac{1}{2}}(S),H^{\frac{1}{2}}_0(S)}.\end{aligned}$$
Combining this with \eqref{t2f} and \eqref{g4}, we obtain
\bel{t2g}\abs{\int_\Omega (q_{1,\lambda_R}-q_{2,\lambda_R})|\partial_{x_k}H(\cdot,y_\tau)|^2dx} \leq C\left[\max(1,\tau^{2-\frac{n}{2}})+\norm{\mathcal D^1_{S,\lambda_R}-\mathcal D^2_{S,\lambda_R}}_{\mathcal B(H^{\frac{1}{2}}_0(S),H^{-\frac{1}{2}}(S))}\tau^{-n}\right],\ee
where $C>0$ is a constant depending only on $R$, $\chi$, $\kappa$ and $\Omega$. Now let us consider $q=q_{1,\lambda_R}-q_{2,\lambda_R}$ and recall that $q\in \mathcal C^1(\overline{\Omega})$ with
$$\partial_{x_i}q=G_1''(v_{1,\lambda_R})\partial_{x_i}v_{1,\lambda_R}-G_2''(v_{2,\lambda_R})\partial_{x_i}v_{2,\lambda_R},\quad i=1,\ldots,n.$$
On the other hand, applying \eqref{t2a} and following the argumentation after \cite[formula (8.14) pp. 296]{LU} and  \cite[Theorem 9.3, 9.4]{GT}, we deduce that there exists $M_1>0$ depending only on $\kappa$, $\Omega$, $R$ and $\chi$ such that
$$\norm{v_{j,\lambda_R}}_{W^{1,\infty}(\Omega)}\leq M_1.$$
Therefore, there exists $C>0$ depending only on $\kappa$, $\Omega$, $R$ and $\chi$ such that
$$\norm{q}_{W^{1,\infty}(\Omega)}\leq C.$$
Since $\Omega$ is $\mathcal C^{2+\alpha}$, we can extend $q$ into $\tilde{q}\in \mathcal C^1(\R^n)$ satisfying
$$\norm{\tilde{q}}_{W^{1,\infty}(\R^n)}\leq C'\norm{q}_{W^{1,\infty}(\Omega)},$$
with $C'>0$ depending only on $\Omega$. Then, we deduce that there exists $C>0$ depending only on $\kappa$, $\Omega$, $R$ and $\chi$ such that
\bel{t2h}\norm{\tilde{q}}_{W^{1,\infty}(\R^n)}\leq C.\ee
Moreover, we have
$$q(x)=q(x_0)+\int_0^1\nabla \tilde{q}(x_0+s(x-x_0))\cdot (x-x_0)ds=q(x_0)+Q(x)\cdot(x-x_0),\quad x\in\overline{\Omega}$$
and using the fact that 
$$q(x_0)=G_1'(v_{1,\lambda_R}(x_0))-G_2'(v_{2,\lambda_R}(x_0))=G_1'(\lambda_R\chi(x_0))-G_2'(\lambda_R\chi(x_0))=G'(\lambda_R),$$
we obtain 
$$q(x)=G'(\lambda_R)+Q(x)\cdot(x-x_0),\quad x\in\overline{\Omega}.$$
It follows that
$$\int_\Omega (q_{1,\lambda_R}-q_{2,\lambda_R})|\partial_{x_k}H(\cdot,y_\tau)|^2dx=G'(\lambda_R)\int_\Omega |\partial_{x_k}H(\cdot,y_\tau)|^2dx+\int_\Omega Q(x)\cdot(x-x_0)|\partial_{x_k}H(\cdot,y_\tau)|^2dx$$
and applying \eqref{t2h}, we get
$$\abs{\int_\Omega (q_{1,\lambda_R}-q_{2,\lambda_R})|\partial_{x_k}H(\cdot,y_\tau)|^2dx}\geq G'(\lambda_R)\int_\Omega |\partial_{x_k}H(\cdot,y_\tau)|^2dx-C\int_\Omega |x-x_0||\partial_{x_k}H(\cdot,y_\tau)|^2dx.$$
Moreover, applying \eqref{est3}, we obtain
$$\begin{aligned}\int_\Omega |x-x_0||\partial_{x_k}H(\cdot,y_\tau)|^2dx&\leq C\int_\Omega |x-x_0||x-y_{\tau}|^{2-2n}dx\\
\ &\leq C\int_\Omega (|x-y_{\tau}|+|x_0-y_\tau|)|x-y_{\tau}|^{2-2n}dx\\
&\leq C\left(\int_\Omega |x-y_{\tau}|^{3-2n}dx+\tau\int_\Omega |x-y_{\tau}|^{2-2n}dx\right)\end{aligned}$$
and repeating the arguments used in Proposition \ref{p2}, we obtain
$$\abs{\int_\Omega (q_{1,\lambda_R}-q_{2,\lambda_R})|\partial_{x_k}H(\cdot,y_\tau)|^2dx}\geq G'(\lambda_R)\int_\Omega |\partial_{x_k}H(\cdot,y_\tau)|^2dx-C\tau^{3-n},$$
with $C>0$ depending only on $\kappa$, $\Omega$, $R$ and $\chi$. Combining this with \eqref{t2g}, we find
$$G'(\lambda_R)\int_\Omega |\partial_{x_k}H(\cdot,y_\tau)|^2dx \leq C\left[\max(1,\tau^{2-\frac{n}{2}})+\tau^{3-n}+\norm{\mathcal D^1_{S,\lambda_R}-\mathcal D^2_{S,\lambda_R}}_{\mathcal B(H^{\frac{1}{2}}_0(S),H^{-\frac{1}{2}}(S))}\tau^{-n}\right].$$
Taking the sum of the above expression with respect to $k=1,\ldots,n$, we get 
$$G'(\lambda_R)\int_\Omega |\nabla_x H(\cdot,y_\tau)|^2dx \leq C\left[\max(1,\tau^{2-\frac{n}{2}})+\tau^{3-n}+\norm{\mathcal D^1_{S,\lambda_R}-\mathcal D^2_{S,\lambda_R}}_{\mathcal B(H^{\frac{1}{2}}_0(S),H^{-\frac{1}{2}}(S))}\tau^{-n}\right].$$
Then, in similar way to Theorem \ref{t1}, applying \eqref{est3}, we find
$$G'(\lambda_R)\tau^{2-n}\leq C\left[\max(1,\tau^{2-\frac{n}{2}})+\tau^{3-n}+\norm{\mathcal D^1_{S,\lambda_R}-\mathcal D^2_{S,\lambda_R}}_{\mathcal B(H^{\frac{1}{2}}_0(S),H^{-\frac{1}{2}}(S))}\tau^{-n}\right],\quad\tau\in(0,\delta).$$
Dividing this inequality by $\tau^{2-n}$, we obtain
$$G'(\lambda_R)\leq C\left[\tau+\norm{\mathcal D^1_{S,\lambda_R}-\mathcal D^2_{S,\lambda_R}}_{\mathcal B(H^{\frac{1}{2}}_0(S),H^{-\frac{1}{2}}(S))}\tau^{-2}\right],\quad \tau\in(0,\delta).$$
From this last estimate and \eqref{t2a} by choosing $\tau=\left(\norm{\mathcal D^1_{S,\lambda_R}-\mathcal D^2_{S,\lambda_R}}_{\mathcal B(H^{\frac{1}{2}}_0(S),H^{-\frac{1}{2}}(S))}\right)^{\frac{1}{3}}$ when $\norm{\mathcal D^1_{S,\lambda_R}-\mathcal D^2_{S,\lambda_R}}_{\mathcal B(H^{\frac{1}{2}}_0(S),H^{-\frac{1}{2}}(S))}$ is sufficiently small, one can easily check that there exists $C>0$ depending only on $\kappa$, $\Omega$, $R$ and $\chi$, such that the following estimate 
\bel{t2j}G'(\lambda_R)\leq C\norm{\mathcal D^1_{S,\lambda_R}-\mathcal D^2_{S,\lambda_R}}_{\mathcal B(H^{\frac{1}{2}}_0(S),H^{-\frac{1}{2}}(S))}^{\frac{1}{3}}\ee
holds true. Combining this with the fact that, according to \eqref{t2a}, we have
$$|G(\lambda)|\leq |G(0)|+|\lambda|\norm{G'}_{L^\infty(-R,R)}\leq RG'(\lambda_R),\quad \lambda\in[-R,R],$$
we deduce \eqref{t2e} from \eqref{t2j}. This completes the proof of the theorem.

\bigskip
\vskip 1cm

\end{document}